\documentclass[a4paper,12pt]{amsart}
    \usepackage{hyperref}
    \usepackage[T1]{fontenc}
    \usepackage{mathtools,amssymb,amsthm}
    \usepackage{newtxtext,newtxmath}
    \usepackage{microtype}
    \usepackage{enumitem}
    \usepackage[
      a4paper,
      top=2.3cm,
      left=2.5cm,
      right=2.5cm,
      bottom=2.8cm,
      heightrounded,
      centering
    ]{geometry}
    
    \usepackage{multicol} 
    \usepackage[usenames,dvipsnames]{xcolor} 
    \usepackage{tikz, tikz-3dplot, pgfplots}
    \usepackage{tikz-cd}
    \usetikzlibrary[positioning, patterns] 
    \usetikzlibrary{matrix, arrows, decorations.pathmorphing}

    \definecolor{light-gray}{gray}{0.7}

    \usepackage{calligra}
    \usepackage{mathrsfs}
    \DeclareMathAlphabet{\mathcalligra}{T1}{calligra}{m}{n}
    \DeclareFontShape{T1}{calligra}{m}{n}{<->s*[1.5]callig15}{}

    \newtheorem{theorem}{Theorem}[section]

    \newtheorem{lemma}[theorem]{Lemma}
    \newtheorem{proposition}[theorem]{Proposition}

    \theoremstyle{definition}

    \newtheorem{theorem-definition}[theorem]{Theorem-Definition}
    \newtheorem{lemma-definition}[theorem]{Lemma-Definition}

    \numberwithin{equation}{section}
    
    \DeclareMathOperator{\Hilb}{Hilb}
    
    \DeclareMathOperator{\Pic}{Pic}
    \DeclareMathOperator{\Perf}{Perf}
    \DeclareMathOperator{\PBs}{PBs}
    \DeclareMathOperator{\Bs}{Bs}
    \DeclareMathOperator{\Supp}{Supp}
    
    \DeclareMathOperator{\Hom}{Hom}
    \DeclareMathOperator{\rk}{rk}
    
    \newcommand{\Db}{\mathrm D^{\mathrm b}_{\mathrm{coh}}}
    \newcommand{\cA}{\mathcal A}
    \newcommand{\cB}{\mathcal B}
    \newcommand{\cO}{\mathcal O}
    \newcommand{\cE}{\mathcal E}
    \newcommand{\cI}{\mathcal I}
    \newcommand{\cJ}{\mathcal J}
    \newcommand{\cL}{\mathcal L}
    \newcommand{\bP}{\mathbb P}
    \newcommand{\Jbar}{\overline J}
    \newcommand{\Picbar}{\overline{\Pic}}
    
    \newcommand {\cHom}{\mathscr{H}\kern-5pt\mathcalligra{om}}

    \title[Semiorthogonal indecomposability for Hilbert schemes]
    {Semiorthogonal indecomposability for Hilbert schemes of points on integral locally planar curves}
    \author[Qingyuan Jiang]{Qingyuan Jiang}
    \address{Department of Mathematics, The Hong Kong University of Science and Technology,
    Clearwater Bay, Kowloon, Hong Kong}
    \email{jiangqy@ust.hk}
    \author[Xun Lin]{Xun Lin}
    \address{Department of Mathematics, The Hong Kong University of Science and Technology,
    Clearwater Bay, Kowloon, Hong Kong}
    \email{linx@ust.hk}
    

\begin{document}
    
    \begin{abstract}
    Let $C$ be an integral projective curve of arithmetic genus $g$ with locally planar singularities over an algebraically closed field.
    We prove that for every $1\leq n\leq g-1$, both $\Perf(\Hilb^n(C))$ and $\Db(\Hilb^n(C))$ are semiorthogonally indecomposable.
    We also establish the corresponding relative $S$-linear statement for a flat family of such curves  over a connected base $S$, with admissibility required in the $\Db$ case.
    Our results hold in arbitrary characteristic.
    \end{abstract}

    \maketitle
    
    \section{Introduction}
    Semiorthogonal decompositions often reflect birational constructions. Their nonexistence is therefore a meaningful form of categorical minimality, conjecturally closely related to the birational geometry of the variety. A useful geometric obstruction was introduced by Kawatani and Okawa \cite{KawataniOkawa}: for a smooth projective variety, sections of the canonical bundle strongly constrain the supports of the components of a semiorthogonal decomposition.  Lin refined this point of view by replacing the canonical base locus with the \emph{paracanonical base locus}
    \[
      \PBs|\omega_X|
      :=\bigcap_{L\in\Pic^0(X)(k)}\Bs|\omega_X\otimes L|.
    \]
    In particular, Lin proved that the derived category of the $n$-th symmetric product of a smooth curve of genus $g\geq2$ is indecomposable for $1 \leq n\leq g-1$ \cite{Lin}.  See also \cite{Caucci} for a relative interpretation of this locus through the Albanese morphism. In a related direction, Lin and Yu proved the indecomposability of the bounded derived categories of Brill--Noether varieties for a general smooth curve in the corresponding range \cite{LinYu}.
    
    The purpose of this note is to establish the analogous statement for Hilbert schemes of points on singular curves.  Our main result is the following:
    
    \begin{theorem}\label{thm:main}
    Let $C$ be an integral projective curve of arithmetic genus $g$ with locally planar singularities over an algebraically closed field $k$. For every integer $1\leq n\leq g-1$, neither $\Perf\bigl(\Hilb^n(C)\bigr)$ nor $\Db(\Hilb^n(C))$ admits a nontrivial semiorthogonal decomposition.
    \end{theorem}
    
    Theorem~\ref{thm:main} settles the absolute case of the indecomposability question posed in \cite[Section~4.2.3]{JiangAbel}. In families, we establish the corresponding $S$-linear statement for $\Perf$, together with its analogue for admissible $S$-linear decompositions of $\Db$; see Theorem~\ref{thm:relative}.
    
    The stated range is empty when $g\leq1$.  Hence, for the remainder of this paper, we may and do assume $g\geq2$. 
    
    There are two points to address.  
    First, $\Hilb^n(C)$ is generally singular. Lin observed that Spence's Cohen--Macaulay argument, together with the rigidity of semiorthogonal components under algebraically trivial twists, gives the corresponding singular paracanonical support theorem \cite[Section~3.4]{Lin}. We record the empty-locus case in Proposition~\ref{prop:singular-criterion}, emphasizing that the argument is valid in arbitrary characteristic.
    
    Second, one needs enough paracanonical sections on the Hilbert scheme.  Fix a smooth point $p\in C$, let $\Jbar=\Picbar^0(C)$, and choose a universal rank-one torsion-free sheaf $\cJ$ on $C\times\Jbar$, rigidified along $p\times\Jbar$.  The Abel-map construction of \cite{JiangAbel} (see also \cite{JiangLeung} for the smooth case), based on derived projectivizations of  complexes \cite{JiangProjectivizations} , identifies $\Hilb^n(C)$ with the derived projectivization of the perfect complex
    \[
      \cE_n
      :=\left(Rq_*\bigl(\cJ\otimes r^*\cO_C(np)\bigr)\right)^\vee
    \]
    on $\Jbar$, where $q$ and $r$ are the two projections.  Although $\cE_n$ can have negative virtual rank, its derived projectivization is classical when $C$ is locally planar.  This derived description contains information that is not easily visible from a purely classical perspective.  It yields the formula
    \begin{equation}\label{eq:intro-canonical}
      \omega_{\Hilb^n(C)}
      \simeq
      \cO_{\Hilb^n(C)}\bigl((g-1-n)D_p\bigr)
      \otimes a_n^*\cO_{\Jbar}(\Theta_p),
    \end{equation}
    where $D_p$ is the divisor of subschemes containing $p$, $a_n$ is the Abel map, and $\Theta_p$ is the theta divisor.
    
    Given $Z\in\Hilb^n(C)$, we choose $p\notin\Supp(Z)$.  A cohomological argument, modeled on \cite[Proposition~(3.5)]{AltmanKleiman}, translates $\Theta_p$ so that it avoids $a_n(Z)$.  Multiplying the translated theta section by the canonical section of $(g-1-n)D_p$ produces a paracanonical section nonvanishing at $Z$.  Consequently
    \[
      \PBs\bigl|\omega_{\Hilb^n(C)}\bigr|=\varnothing,
    \]
    and the singular paracanonical criterion proves that $\Perf\bigl(\Hilb^n(C)\bigr)$ is semiorthogonally indecomposable.

    The passage from $\Perf$ to $\Db$ uses a result of Kuznetsov and Shinder: for every projective scheme over a perfect field, $\Perf(X)$ is semiorthogonally indecomposable if and only if $\Db(X)$ is \cite[Corollary~6.6]{KuznetsovShinder}. 
    Since every algebraically closed field is perfect, this completes the proof of Theorem~\ref{thm:main}.
    
    \medskip
    The result is sharp.  Indeed, the locally planar specialization of the semiorthogonal decomposition in \cite{JiangAbel} gives, for $n\geq g$,
    \[
      \mathrm D \bigl(\Hilb^n(C)\bigr)
      =\left\langle
          \mathrm D\bigl(\Hilb^{2g-2-n}(C)\bigr),  \text{$(1-g+n)$ copies of } \mathrm D\bigl(\Picbar^n(C)\bigr)
        \right\rangle,
    \]
    with the usual convention that $\Hilb^{m}(C)$ is empty for $m<0$.
    
    No characteristic-zero assumption is required in our proof.  The conditions in the classical results on compactified Jacobians, the projectivization result for Abel maps \cite{JiangAbel}, the Kawatani--Okawa theorem \cite{KawataniOkawa}, Spence's argument (see Proposition~\ref{prop:singular-criterion}), and the fiberwise criterion (Lemma~\ref{lem:fiberwise-criterion}) are all valid over an algebraically closed field of arbitrary characteristic.  
    
    \subsection*{Acknowledgement}
    This work was supported by the Hong Kong Research Grants Council ECS grant (Grant No. 26311724) and GRF grant (Grant No. 16311625).

    \section{Abel maps as projectivizations}\label{sec:abel}
    
    Throughout this section, $C$ is as in Theorem~\ref{thm:main}. In particular, $C$ is Gorenstein.  The Hilbert scheme $C^{[n]}:=\Hilb^n(C)$ is an integral projective local complete intersection of dimension $n$; see \cite[Theorem~(6)]{AltmanIarrobinoKleiman}.  We use Grothendieck's quotient convention for projectivizations.
    
    Fix a smooth point $p\in C_{\mathrm{sm}}(k)$, set $\Jbar=\Picbar^0(C)$, and let $q:C\times\Jbar\to\Jbar$ and $r:C\times\Jbar\to C$ be the projections. Since $C$ is integral, $\Jbar$ is a fine compactified Jacobian; in particular, after rigidification along $p\times\Jbar$, it carries a universal sheaf $\cJ$. Set
    \begin{equation}\label{eq:En}
      \cE_n
      :=\left(Rq_*\bigl(\cJ\otimes r^*\cO_C(np)\bigr)\right)^\vee
      \in\Perf(\Jbar).
    \end{equation}
    By \cite[Corollary~2.10]{JiangAbel}, it is a perfect complex, whose rank is given by Riemann--Roch:
    \begin{equation}\label{eq:rank}
      \rk(\cE_n)=n+1-g.
    \end{equation}
    (More precisely, $\cJ\otimes r^*\cO_C(np)$ is $q$-flat, hence perfect relative to $\Jbar$.  Thus \cite[Corollary~2.10]{JiangAbel}, applied to $(\cO_{C\times\Jbar},\cJ\otimes r^*\cO_C(np))$, gives a perfect connective $\mathcal Q$-complex, of Tor-amplitude contained in $[0,1]$ in the homological convention of \cite{JiangAbel}, and canonical equivalences
    \[
      Rq_*\bigl(\cJ\otimes r^*\cO_C(np)\bigr)
      \simeq   \mathcal Q_q\bigl(\cO_{C\times\Jbar}, \cJ\otimes r^*\cO_C(np)\bigr)^\vee,
      \qquad
      \cE_n\simeq  \mathcal Q_q\bigl(\cO_{C\times\Jbar},
          \cJ\otimes r^*\cO_C(np)\bigr).)
          \]
    
    For a length-$n$ subscheme $Z\subset C$, let $\cI_Z$ be its ideal sheaf and set $\cL_Z:=\mathcal{H}\!om(\cI_Z,\cO_C)$.  Since $C$ is Gorenstein, $\cL_Z$ is rank-one torsion-free of degree $n$.  We normalize the Abel map as
    \begin{equation}\label{eq:abel}
      a_n \colon C^{[n]}\to\Jbar,
      \qquad
      Z\longmapsto\cL_Z(-np).
    \end{equation}
    The construction of \cite[Section~4.2.2 and Corollary~4.6]{JiangAbel}, refining the classical description in \cite[Lemma~(5.17)]{AltmanKleiman}, gives an isomorphism
    \begin{equation}\label{eq:projectivization}
      C^{[n]}\simeq\bP_{\Jbar}(\cE_n)
    \end{equation}
    over $\Jbar$.  Here, $\bP_{\Jbar}(\cE_n)$ is classical for all $n\ge 0$ (\cite[Lemma~4.7]{JiangAbel}).
    
    Let
    \[
      D_p:=\{Z\in C^{[n]}:p\in Z\}.
    \]
    Because $p$ is smooth, this is the effective Cartier divisor obtained by adding $p$ to a length-$(n-1)$ subscheme.  The rigidification of $\cJ$ and the exact sequence
    \[
      0\to\cO_C((n-1)p)
       \to\cO_C(np)
       \to k(p)\to 0
    \]
    gives a distinguished triangle
    \begin{equation}\label{eq:triangle}
    \cO_{\Jbar}\to \cE_n\to\cE_{n-1}
      \to \cO_{\Jbar}[1].
    \end{equation}
    The composite of the first arrow with the universal quotient on $\bP(\cE_n)$ is a section of $\cO_{\bP(\cE_n)}(1)$.  Its vanishing locus is $\bP(\cE_{n-1})$ (\cite[Corollary~4.32]{JiangProjectivizations}), which is $D_p$ under \eqref{eq:projectivization}.  Hence
    \begin{equation}\label{eq:O1}
      \cO_{\bP(\cE_n)}(1)\simeq\cO_{C^{[n]}}(D_p).
    \end{equation}
    
    Taking determinants of \eqref{eq:triangle} yields $\det(\cE_n)\simeq\det(\cE_{n-1})$.  When $n=g-1$, the complex $\cE_{g-1}$ has virtual rank zero, and its determinant of cohomology carries a canonical theta section. Let
    \begin{equation}\label{eq:theta}
      \cO_{\Jbar}(\Theta_p):=\det(\cE_{g-1}),
      \quad \text{where} \quad
      \Theta_p
      =\left\{[I]\in\Jbar:
          H^0\bigl(C,I((g-1)p)\bigr)\neq0
        \right\}.
    \end{equation}
    To check the determinant convention, locally represent $Rq_*(\cJ\otimes r^*\cO_C((g-1)p))$ by a two-term complex $[K^{0}\to K^{1}]$ of vector bundles of equal rank. Its derived dual represents $\cE_{g-1}$, so
    \[
      \det(\cE_{g-1})
      =\det(K^{1})\otimes\det(K^{0})^{-1},
    \]
    and the determinant of $K^{0}\to K^{1}$ is a section of this line bundle whose zero locus is precisely the locus where $H^0$ (equivalently $H^{1}$) is nonzero.
    
    The generalized theta divisor is Cartier; see \cite{Soucaris}.  We obtain
    \begin{equation}\label{eq:detEn}
      \det(\cE_n)\simeq\cO_{\Jbar}(\Theta_p)
      \qquad\text{for all }n\geq0.
    \end{equation}
    
    \begin{proposition}\label{prop:canonical}
    For every $n\geq1$, the dualizing line bundle of $C^{[n]}$ is
    \begin{equation}\label{eq:canonical}
      \omega_{C^{[n]}}
      \simeq
      \cO_{C^{[n]}}\bigl((g-1-n)D_p\bigr)
      \otimes a_n^*\cO_{\Jbar}(\Theta_p).
    \end{equation}
    \end{proposition}
    
    \begin{proof}
    The compactified Jacobian $\Jbar$ is a local complete intersection of dimension $g$ with trivial dualizing sheaf (this holds in arbitrary characteristic; see, e.g., \cite[Theorem~A]{MRV17}). 
    Thus $\det \mathbb{L}_{\Jbar} \simeq \omega_{\Jbar} \simeq \cO_{\Jbar}$.
    Taking determinants of the exact triangle of cotangent complexes gives
    \[
    \det\mathbb{L}_{C^{[n]}}
      \simeq
      a_n^*\det\mathbb{L}_{\Jbar}
      \otimes
      \det\mathbb{L}_{C^{[n]}/\Jbar}.
    \]
    From the Euler triangle for a derived projectivization \cite[Theorem 4.27]{JiangProjectivizations}, we obtain
    \[
      \det\mathbb {L}_{\bP(\cE_n)/\Jbar}
      \simeq a_n^*\det(\cE_n) \otimes\cO_{\bP(\cE_n)}\bigl(-\rk(\cE_n)\bigr).
    \]
    Since the derived projectivization is the classical local complete intersection variety $C^{[n]}$, $\det \mathbb{L}_{C^{[n]}} \simeq \omega_{C^{[n]}}$.  The result now follows from \eqref{eq:rank}, \eqref{eq:O1}, and \eqref{eq:detEn}.
    \end{proof}
    
    We next record a cohomological fact that allows us to move any relevant sheaf off the theta divisor. It is a reformulation of the argument in \cite[Proposition~(3.5)]{AltmanKleiman}.
    
    \begin{lemma}\label{lem:translation}
    Let $F$ be a rank-one torsion-free sheaf on $C$ of degree $e$, where $g-1\leq e\leq2g-2$.  There is a line bundle $M\in\Pic^0(C)(k)$ such that
    \[
      H^1(C,F\otimes M)=0 \qquad\text{and}\qquad h^0(C,F\otimes M)=e+1-g.
    \]
    In particular, if $e=g-1$, then $H^0(C,F\otimes M)=0$.
    \end{lemma}
    
    \begin{proof}
    Let $m=2g-2-e$ and $G_0=F(mp)$.  Choose $M_0\in\Pic^0(C)(k)$ such that $G_0\otimes M_0\not\simeq\omega_C$. Such a choice is possible: if $G_0$ is not invertible, no line-bundle twist of it is isomorphic to $\omega_C$. If it is invertible, the unique excluded choice is $M_0\simeq\omega_C\otimes G_0^{-1}$; since $g\geq 2$, the group $\Pic^0(C)(k)$ has more than one point, so another choice exists.
    
    Serre duality yields
    \[
      H^1(C,G_0\otimes M_0)^\vee  \simeq\Hom(G_0\otimes M_0,\omega_C).
    \]
    A nonzero map between rank-one torsion-free sheaves of the same degree is an isomorphism: it is injective at the generic point, and its zero-dimensional cokernel has length equal to the difference of the degrees.  It follows that $H^1(C,G_0\otimes M_0)=0$.  This is also the degree-$(2g-2)$ case of \cite[Proposition~(3.5)(iii)(g)]{AltmanKleiman}.
    
    Suppose now that $G$ is rank-one torsion-free, $H^1(C,G)=0$, and $\deg G\geq g$.  Riemann--Roch gives $h^0(C,G)>0$.  Choose a nonzero section and then a smooth point $x$ where it does not vanish.  From
    \[
      0\to G(-x)\to G\to k(x)\to0
    \]
    the evaluation map $H^0(C,G)\to k(x)$ is surjective, and therefore $H^1(C,G(-x))=0$.
    
    Starting with $G_0\otimes M_0$, repeat this step $m$ times.  We obtain smooth points $x_1,\ldots,x_m$ such that
    \[
      H^1\bigl(C,F(mp)\otimes M_0(-x_1-\cdots-x_m)\bigr)=0.
    \]
    The line bundle
    \(
      M:=M_0\otimes\cO_C(mp-x_1-\cdots-x_m)
    \)
    has degree zero, hence lies in $\Pic^0(C)(k)$ since $C$ is integral, and $H^1(C,F\otimes M)=0$.  The formula for $h^0$ follows from Riemann--Roch.
    \end{proof}

    \section{Paracanonical sections and indecomposability}
    \label{sec:paracanonical}
    
    We first establish the singular version of Lin's criterion in \cite{Lin} that will be used below. For a projective Gorenstein variety $X$, define
    \[
      \PBs|\omega_X| :=\bigcap_{L\in\Pic^0(X)(k)}\Bs|\omega_X\otimes L|.
    \]
    Thus $x\notin\PBs|\omega_X|$ precisely when some algebraically trivial twist of $\omega_X$ has a section nonvanishing at $x$.
    
    \begin{proposition}[{Lin--Spence's criterion; cf. \cite{Lin}, \cite{Spence}}]\label{prop:singular-criterion}
    Let $X$ be an integral projective Gorenstein variety over an algebraically closed field. If $\PBs|\omega_X|=\varnothing$, then neither $\Perf(X)$ nor $\Db(X)$ admits a nontrivial semiorthogonal decomposition.
    \end{proposition}
    
    \begin{proof}
    Over $\mathbb C$, the assertion for $\Perf(X)$ is already the empty-locus case of Lin's singular paracanonical support theorem \cite[Section~3.4]{Lin}, which builds on Spence's Cohen--Macaulay criterion \cite[Theorem~3.1]{Spence}. We recall the argument to record that it is valid over an arbitrary algebraically closed field.
    
    Suppose $\Perf(X)=\langle\cA,\cB\rangle$. By \cite[Theorem~3.9]{KawataniOkawa}, $\cA\otimes L=\cA$ for every $L\in\Pic^0(X)(k)$.  Since $\cB={}^\perp\cA$ and tensoring by $L$ is an autoequivalence preserving $\cA$, it follows that $\cB\otimes L=\cB$.
    
    We explain the modification of Spence's proof.  Let $x$ be a closed point outside $\PBs|\omega_X|$, and choose $L\in\Pic^0(X)(k)$ and $s\in H^0(X,\omega_X\otimes L)$ with $s(x)\neq0$.  Since $X$ is Cohen--Macaulay, there is a Koszul zero-cycle $Z_x$ supported at $x$, with $\cO_{Z_x}\in\Perf(X)$.  Write its decomposition triangle as
    \[
      B_x\to\cO_{Z_x}\to A_x  \xrightarrow{f}B_x[1],
      \quad\text{where}\quad 
      A_x\in\cA,\quad B_x\in\cB.
    \]
    Multiplication by $s$ and Grothendieck duality give
    \begin{equation*}
      \Hom\bigl(A_x,B_x\otimes\omega_X\otimes L[1]\bigr) 
      \simeq \Hom\bigl(B_x\otimes L,A_x[\dim X-1]\bigr)^\vee=0.
    \end{equation*}
    The last equality uses $B_x\otimes L\in\cB$ and semiorthogonality. Hence the composite of $f$ with multiplication by $s$ is zero.  On an open neighborhood of $x$ where $s$ is invertible, the morphism $f$ vanishes.
    
    The rest of \cite[Theorem~3.1]{Spence} now applies verbatim: every Koszul zero-cycle outside $\PBs|\omega_X|$ belongs to exactly one component, all such cycles belong to the same component because $X$ is integral, and every object of the other component is supported on $\PBs|\omega_X|$.  If this locus is empty, the other component is zero.
    

    Thus $\Perf(X)$ is semiorthogonally indecomposable. Since every algebraically closed field is perfect, \cite[Corollary~6.6]{KuznetsovShinder} implies that $\Db(X)$ is semiorthogonally indecomposable.
    \end{proof}
    
    We now prove the main theorem by showing that the paracanonical base locus is empty.
    
    \begin{proposition}\label{prop:pbs-empty}
    For $1\leq n\leq g-1$, one has
    \[
      \PBs  |\omega_{C^{[n]}}|=\varnothing.
    \]
    \end{proposition}
    
    \begin{proof}
    The strategy is the singular-curve analogue of Lin's argument for symmetric products of smooth curves \cite[Section~4.1]{Lin}: we move the incidence divisor and the theta divisor independently so that both avoid a prescribed point.
    
    Fix $Z\in C^{[n]}$.  Choose a smooth point $p\notin\Supp(Z)$, and use this point in the constructions of Section~\ref{sec:abel}; equivalently, use \eqref{eq:abel}--\eqref{eq:canonical} with this choice of $p$. Let
    \[
      F_Z:=\cL_Z\bigl((g-1-n)p\bigr).
    \]
    This is rank-one torsion-free of degree $g-1$.  By Lemma~\ref{lem:translation}, there is $M\in\Pic^0(C)(k)$ such that
    \begin{equation}\label{eq:vanishing-at-Z}
      H^0(C,F_Z\otimes M)=0.
    \end{equation}
    
    Let $t_M:\Jbar\to\Jbar$ be tensorization by $M$.  From \eqref{eq:abel}, \eqref{eq:theta}, and \eqref{eq:vanishing-at-Z}, we obtain
    \(
      t_M(a_n(Z))\notin\Theta_p.
    \)
    Let $\vartheta_p$ be the canonical section of $\cO_{\Jbar}(\Theta_p)$ and $\sigma_p$ the canonical section of $\cO_{C^{[n]}}(D_p)$.  Then
    \begin{equation}\label{eq:paracanonical-section}
      \sigma_p^{\,g-1-n}\cdot a_n^*(t_M^*\vartheta_p)
    \end{equation}
    is a section of
    \[
      \cO_{C^{[n]}}\bigl((g-1-n)D_p\bigr)
      \otimes a_n^*t_M^*\cO_{\Jbar}(\Theta_p).
    \]
    It does not vanish at $Z$: the first factor is nonzero because $p\notin\Supp(Z)$, and the second is nonzero by the choice of $M$.
    
    The line bundle
    \[
      Q_M:=t_M^*\cO_{\Jbar}(\Theta_p)
      \otimes\cO_{\Jbar}(\Theta_p)^{-1}
    \]
    is algebraically trivial.  
    Indeed, let  $\mathrm{act}\colon \Pic^0(C)\times\Jbar\to\Jbar$,  $(M,I)\mapsto I\otimes M$ be the tensor action and consider the line bundle
    \[
    \mathrm{act}^*\cO_{\Jbar}(\Theta_p)
    \otimes\operatorname{pr}_2^*\cO_{\Jbar}(\Theta_p)^{-1}.
    \]
    Its restriction over $M$ is $Q_M$, while its restriction over the identity is trivial.  
    Since $\Pic^0(C)$ is connected, this shows $Q_M$ is  algebraically trivial.
    Pulling the same family back along
    \[\mathrm{id} \times a_n:
      \Pic^0(C)\times C^{[n]}\to 
      \Pic^0(C)\times\Jbar
    \]
    shows that $P_M:=a_n^*Q_M$ lies in $\Pic^0(C^{[n]})(k)$.
    By Proposition~\ref{prop:canonical}, the section \eqref{eq:paracanonical-section} belongs to
    \[
      H^0\bigl(C^{[n]},\omega_{C^{[n]}}\otimes P_M\bigr)
    \]
    and is nonzero at $Z$.  Since $Z$ was arbitrary, the paracanonical base locus is empty.
    \end{proof}
    
    \begin{proof}[Proof of Theorem~\ref{thm:main}]
    The scheme $C^{[n]}$ is integral, projective, and a local complete intersection, hence Gorenstein.  Combining Proposition~\ref{prop:singular-criterion} and  Proposition~\ref{prop:pbs-empty} yields the result.
    \end{proof}

    \section{Relative Semiorthogonal Indecomposability}
    \label{sec:relative}
    We conclude with the relative version of Theorem~\ref{thm:main}. 
    Related approaches include Okawa's study of $f$-linear semiorthogonal decompositions via relative canonical base loci \cite{OkawaIrregular}, Pirozhkov's notion of noncommutative stable semiorthogonal indecomposability \cite{Pirozhkov}, and the moduli-theoretic study of semiorthogonal decompositions in smooth proper families by Belmans, Okawa, and Ricolfi \cite{BelmansOkawaRicolfi}. Our argument instead uses a direct fiberwise criterion.

    Throughout this section, we fix a connected separated noetherian base scheme $S$ of finite Krull dimension.
    Let $f \colon X \to S$ be a flat morphism. For any geometric point $\overline{s}$ of $S$, we form the Cartesian diagram
    \begin{equation*}
    	\begin{tikzcd}		
    	X_{\overline{s}} =X \times_S \{\overline{s}\} \ar{d}{f_{\overline{s}}} \ar{r}{i_{X_{\overline{s}}}} & X \ar{d}{f} \\
    		\{\overline{s}\} \ar{r}{i_{\overline{s}}} & S,
    	\end{tikzcd}
    	\end{equation*}
    and for an object $F \in \Perf(X)$, we write $F|_{X_{\overline{s}}} := L i_{X_{\overline{s}}}^* F$ for its derived pullback.
   
    Recall that a full triangulated subcategory of $\Perf(X)$ or $\Db(X)$ is called \emph{$S$-linear} if it is stable under tensoring with derived pullbacks of objects of $\Perf(S)$.     

    \begin{lemma}[Fiberwise criterion]
    \label{lem:fiberwise-criterion}
    Let $f\colon X\to S$ be a faithfully flat proper morphism.
    If $\Perf(X_{\overline{s}})$ admits no  nontrivial semiorthogonal decomposition for every geometric point $\overline{s} \to S$,  then $\Perf(X)$ admits no nontrivial $S$-linear semiorthogonal decomposition. Moreover, $\Db(X)$ admits no nontrivial $S$-linear semiorthogonal decomposition with admissible components. 
    \end{lemma}
    
    \begin{proof}
    Suppose $\Perf(X)=\langle\cA,\cB\rangle$ is an $S$-linear semiorthogonal decomposition. Since $X$ is quasi-compact and quasi-separated, there exists a perfect complex $P$ which generates $\mathrm{D_{qc}}(X)$ and thickly generates $\Perf(X)$; see \cite[Theorem~3.1.1]{BondalVanDenBergh}. Let 
    \[ P_{\cB} \to P \to P_{\cA} \to P_{\cB}[1]
    \]
    be the decomposation triangle associated with $\Perf(X)=\langle\cA,\cB\rangle$. Since the projection functors are exact and preserve direct summands, $\cA$ (resp. $\cB$) is thickly generated by $P_{\cA}$ (resp. $P_{\cB}$).
    
    For any geometric point $i_{\overline{s}} \colon \overline{s} \to S$, faithfully flat base-change (\cite[Proposition~3.15]{BLM+}, \cite[Proposition~5.1]{KuznetsovBaseChange}) yields a semiorthogonal decomposition
    \[
    \Perf(X_{\overline{s}}) = \langle \cA_{\overline{s}}, \cB_{\overline{s}}\rangle,
    \]
    where $\cA_{\overline{s}}$ (resp. $\cB_{\overline{s}}$) is the full triangulated subcategory of $\Perf(X_{\overline{s}})$ thickly generated by $(P_{\cA})|_{X_{\overline{s}}}$ (resp. $(P_{\cB})|_{X_{\overline{s}}}$). Therefore, the condition of the lemma implies that precisely one of $\cA_{\overline{s}}$ and $\cB_{\overline{s}}$ is  zero for each $\overline{s}$. Denote $Z_{\cA}$ (resp. $Z_{\cB}$) the subset of $S$ consisting of $s\in S$ such that $\cA_{\overline{s}}\not\simeq 0$ (resp. $\cB_{\overline{s}} \not\simeq 0$ ). We thus obtain a disjoint union $S = Z_{\cA} \sqcup Z_{\cB}$. 
    
    On the other hand, $Z_{\cA}$ (resp.\ $Z_{\cB}$) coincides with the image under $f$ of the support of the perfect complex $P_{\cA}$ (resp.\ $P_{\cB}$). Since $f$ is proper, both subsets are closed. As $S$ is connected, it follows that $Z_{\cA} = \varnothing$ or $Z_{\cB} = \varnothing$, and consequently $\cA \simeq 0$ or $\cB \simeq 0$.
    
    Now suppose that $\Db(X)=\langle\cA,\cB\rangle$ is an admissible $S$-linear semiorthogonal decomposition. Since it is strong, \cite[Lemma~3.10]{BLM+} induces
    \[\Perf(X)=\langle\cA_{\mathrm{perf}},\cB_{\mathrm{perf}}\rangle, \quad\text{where}\quad
    \cA_{\mathrm{perf}}:=\cA\cap\Perf(X),
      \quad  \cB_{\mathrm{perf}}:=\cB\cap\Perf(X).
    \]
    The previous result for perfect complexes shows that $\cA_{\mathrm{perf}}\simeq0$ or $\cB_{\mathrm{perf}}\simeq 0$. 
    
    It remains to prove that, for any admissible $S$-linear subcategory $\mathcal{C}\subseteq\Db(X)$, the vanishing of
    $\mathcal{C}_{\mathrm{perf}}
      :=
      \mathcal{C}\cap\Perf(X)$
    implies $\mathcal{C} \simeq 0$. Let $i_{\mathcal{C}}\colon\mathcal{C}\hookrightarrow\Db(X)$ denote the inclusion, with left adjoint $i_{\mathcal{C}}^*$. 
    Applying \cite[Lemma~3.10]{BLM+} to the strong $S$-linear decomposition $\Db(X)=\langle\mathcal C,{}^\perp\mathcal C\rangle$ and using the uniqueness of decomposition triangles, we obtain
    \[
    i_{\mathcal{C}}^*P
      \in
      \mathcal{C}_{\mathrm{perf}} \simeq 0.
    \]
    Therefore, $i_{\mathcal{C}}^*P \simeq 0$. Thus, for every $E\in\mathcal{C}$ and every $j\in\mathbb Z$, we have
    \[
      \Hom_{\Db(X)}(P,E[j])
      \simeq
      \Hom_{\mathcal{C}}\bigl(i_{\mathcal{C}}^*P,E[j]\bigr)
      \simeq
      0.
    \]
    Since $P$ generates $\mathrm D_{\mathrm{qc}}(X)$, this implies $E\simeq 0$. Therefore $\mathcal{C}\simeq 0$.
    \end{proof}
    
    Let $\pi\colon\mathscr{C}\to S$ be a flat projective family of geometrically integral locally planar curves of arithmetic genus $g \ge 1$. 
    We denote by $\Hilb^n_{\mathscr{C}/S}$ the relative Hilbert scheme parametrizing closed subschemes $Z\subseteq\mathscr{C}_T$ that are finite and flat of degree $n$ over an $S$-scheme $T$.
    
    \begin{theorem}\label{thm:relative}
    For every $1\leq n\leq g-1$,  $\Perf(\Hilb_{\mathscr{C}/S}^n)$ admits no nontrivial $S$-linear semiorthogonal decomposition, and  $\Db(\Hilb_{\mathscr{C}/S}^n)$ admits no nontrivial $S$-linear semiorthogonal decomposition with admissible components.
    \end{theorem}
    \begin{proof}
     The relative Hilbert scheme $\Hilb_{\mathscr{C}/S}^n$ is projective over $S$, and its formation commutes with arbitrary base change. Moreover, by \cite[Lemma~4.7]{JiangAbel}, following \cite{AltmanIarrobinoKleiman}, $\Hilb_{\mathscr{C}/S}^n \to S$ is a flat and locally complete intersection morphism of relative dimension $n$. Hence, for every geometric point $\overline{s} \to S$, we have
    \[(\Hilb_{\mathscr{C}/S}^n)_{\overline{s}} \simeq \mathscr{C}_{\overline{s}}^{[n]}.\]
    The result now follows from Theorem~\ref{thm:main} and Lemma~\ref{lem:fiberwise-criterion}.
    \end{proof}


    \end{document}